\documentclass[twocolumn]{autart}    

\usepackage{graphicx}          
\usepackage{amssymb,amsmath}
\usepackage{color}
\usepackage{xcolor}
\usepackage{theoremref}
\usepackage{comment}
\newenvironment{proof}{\textbf{Proof}:}{\hfill$\square$}

\newtheorem{definition}{Definition}
\newtheorem{proposition}{Proposition}

\newtheorem{corollary}{Corollary}
\newtheorem{lemma}{Lemma}

\newtheorem{theorem}{Theorem}

\newcommand{\dn}{\mathbf{d}}
\newcommand{\R}{\mathbb{R}}

\usepackage[colorlinks,
linkcolor=blue,
anchorcolor=blue,
citecolor=blue
]{hyperref}
\usepackage[round]{natbib}

\begin{document}

\begin{frontmatter}

\title{Robust Finite-time Stabilization of Linear Systems with Limited State Quantization} 

\thanks[footnoteinfo]{This paper was not presented at any IFAC 
meeting. }

\author[Inria]{Yu Zhou}\ead{yu.zhou@inria.fr},    
\author[Inria]{Andrey Polyakov}\ead{andrey.polyakov@inria.fr},               
\author[Inria]{Gang Zheng}\ead{gang.zheng@inria.fr}  

\address[Inria]{Inria, Univ. Lille, CNRS, Centale Lille, Lille, France}  

\begin{abstract}      
 This paper investigates the robust asymptotic stabilization of a linear time-invariant (LTI) system by a static feedback with a static state quantization. It is shown that the controllable LTI system 
 can be stabilized to zero in a finite time by means of a nonlinear feedback with a quantizer having a limited (finite) number of values (quantization seeds) even when  all parameters of the controller and the quantizer are time-invariant. The control design is based on generalized homogeneity. A homogeneous spherical quantizer is introduced. The static homogeneous feedback is shown to be local (or global) finite-time stabilizer for the linear system (dependently of the system matrix).  The tuning rules for both the quantizer and the feedback law are obtained in the form of Linear Matrix Inequalities (LMIs). The closed-loop system is proven to be robust with respect to some bounded matched and vanishing mismatched perturbations. Theoretical results are supported by numerical simulations.
\end{abstract}

\end{frontmatter}

\section{Introduction}
Quantization is an important process of networked control system design aimed at the reduction of data transmission \cite{jiang2013survry}, \cite{wang2017encoding}. Quantization involves partitioning the state space into discrete and disjoint subsets known as quantization cells. Within each quantization cell, all states are represented by just one quantization seed (an element of the cell). In network control, a number of the quantization seed is transmitted and stored instead of the exact value of the state vector belonging to the corresponding quantization cell.
Quantizers (quantization operators) can be classified into static (as in \cite{elia_etal_2001_TAC}) and dynamic (see \cite{brockett_Liberzon2000TAC}).
A static quantizer has time-invariant quantization cells and seeds, whereas the quantization cells and seeds of a dynamic quantizer can be adjusted as time evolves. 

The logarithmic quantizer proposed in \cite{elia_etal_2001_TAC} is an example of the static quantizer. The control design and analysis of linear control systems with logarithmic quantizer is based on the quadratic Lyapunov functions. The logarithmic quantizer follows the intuitive idea that the farther state is from the origin, the less precise control action and knowledge about the state are required. Such a quantizer has an infinite (but countable) number of quantization levels/seeds.
A commonly employed dynamic quantizer  \cite{brockett_Liberzon2000TAC} adjusts the quantization precision depending on the state or time. A finite static quantizer refers to a time-invariant quantizer having a finite number of quantization seeds on the whole state space.

A global asymptotic stabilizer designed for a ``non-quantized'' system generally fails to provide asymptotic stability in the case of a finite static quantization \cite{bullo_Liberzon2006TAC}.
Existing studies have addressed asymptotic stabilization only through logarithmic or dynamic  quantizers 
(see also e.g., \cite{Fu_etal_2005_TAC}, \cite{gu2014SCL}, \cite{wang2021TAC}). The former requires infinite number of quantization levels, while the latter needs a transmission of time-varying quantizer parameters for encoding/decoding over the network (e.g., \cite{fu2009finite}, \cite{liu2015dynamic}, \cite{zheng2018quantized}, \cite{wang2022TAC}). 
 A linear feedback with finite static state quantizer can only ensure just a practical stability of LTI system (see, e.g., \cite{delchamps1990TAC}, \cite{bullo_Liberzon2006TAC}, \cite{corradini2008Aut}, \cite{ferranteTAC2015}, \cite{wang2018Aut}). This limitation is due to the fact that a linear feedback cannot reject a nonvanishing quantization errors. 
 Motivated by the above observations, this paper is aimed at  a nonlinear feedback control design for asymptotic (in fact,  finite-time) stabilization of LTI system with a finite static quantization of the  states. The so-called generalized homogeneity is utilized for this purpose.
 
Homogeneity is a symmetry with respect to a group of transformations having topological characterization of a dilation \cite{kawski1991}. Homogeneous systems are widely studied in control systems theory (e.g., \cite{zubov1958systems}, \cite{rosier1992SCL}, \cite{grune2000SIAM},  \cite{bhat_etal_2005_MCSS}, \cite{andrieu2008homogeneous}) due to its many important properties, including the existence of a homogeneous Lyapunov function,  finite/fixed-time convergence dependently on a homogeneity degree, equivalence between local and global stability. 
There are few results on homogeneous controllers with state quantization for sliding mode \cite{yan2019Aut}, \cite{li2017IJRNC} and linear systems \cite{zhou2023arxiv}. However, all previous studies have focused on logarithmic quantizers or dynamic quantizers.
The homogeneity allows a system analysis and control design  to be reduced to a compact (e.g., a sphere) \cite{bhat_etal_2005_MCSS}. Since, for compactly supported states, 
a finite quantization can provide any desired precision, then a compactly designed regulator may become a global stabilizer of the system due to homogeneity.

In our previous work on homogeneous stabilization using a logarithmic quantizer \cite{zhou2023CDC}, \cite{zhou2023arxiv}, we demonstrate the feasibility of finite/fixed-time stabilization of an LTI plant with  states quantization, which, however, requires an infinite number of quantization levels/seeds. Under the limited quantization, the previous approach could only achieve a practical stabilization. 

\vspace{-1mm}
In this paper, a special finite static homogeneous spherical quantizer is developed and combined with homogeneous feedback \cite{polyakov2019IJRNC} to design a finite-time stabilizer for LTI systems. Unlike existing quantizers (uniform, logarithmic, or dynamic) which have bounded quantization cells, each quantization cell of the proposed homogeneous spherical quantizer is a homogeneous cone (unbounded set), but all quantization seeds lie on a unit sphere. This provides a novel solution for stabilization with limited data. We also demonstrate that the proposed quantized stabilizer rejects some classes of additive non-vanishing perturbations.

\vspace{-1mm}
The paper is organized as follows. In Section \ref{sec:problem}, we formulate the problem. Section \ref{sec:pre} provides a brief introduction to homogeneous control systems analysis and design. In Section \ref{sec:main}, we present the main results, which include the design of a homogeneous spherical quantizer as well as  homogeneous quantized feedback. Finally, Section \ref{sec:sim} presents the simulation results.

\subsection*{Notations}\label{section:2}
$\mathbb{R}$: real numbers; $\mathbb{R}_{+}$: non-negative real numbers;
$\mathbb{N}$: natural numbers excluding zero;
$\mathbf{0}$: zero element of a space;  
$P \succ 0 (\prec 0, \succeq 0, \preceq 0)$: positive (negative) definite (or semidefinite) symmetric matrix $P$;
$\lambda_{\min }(P)$, $\lambda_{\max }(P)$: smallest and largest eigenvalues of  matrix $P$;
 $P^{\frac{1}{2}}$: square root matrix (i.e. $(P^{\frac{1}{2}})^{2} = P$);
$|x|=\sqrt{x^{\top}x}$: the conventional Euclidean norm for $x \in \mathbb{R}^n$;
$\|x\| = \sqrt{x^{\top}Px}$: the weighted Euclidean norm of $x \in \mathbb{R}^{n}$ (with $P \succ 0$ to be defined);
$I_n$: the identity $n\times n$ matrix;
$S^{n-1}(1)$: unit sphere $S^{n-1}(1) = \{x\in\mathbb{R}^n:\|x\|=1\}$;
$B(r,x_0)$: open ball with radius $r$ at center $x_0$; $B(r,x_0) = \{x\in\mathbb{R}^n: \|x-x_0\|<r\}$;
$|Q|$: cardinality of a  set $Q$; $\lfloor x\rfloor=\max \{m \in \mathbb{Z} \mid m \leq x\}$: the floor operator for $x\in \R$.

\section{Problem formulation}\label{sec:problem}
In this paper, we consider the system: \vspace{-1mm}
\begin{equation}\label{eq:lin_sys}
	\dot{x} = Ax + Bu + g(t,x),\quad t>0, \quad x(0)=x_0, \vspace{-1mm}
\end{equation}
where $x(t) \in \mathbb{R}^{n}$ is the system state, $u(t)\in \mathbb{R}^{m}$ is the control signal, $A \in \mathbb{R}^{n \times n}$ and $B \in \mathbb{R}^{n \times m}$ are known system matrices, $g:\mathbb{R}_+ \times \mathbb{R}^n \mapsto \mathbb{R}^n$ is an unknown piece-wise continuous function satisfying certain restrictions (see below). 
The pair $(A, B)$ is  controllable.

{   Let $\mathcal{E}$ be a subset of $\mathbb{R}^n$  and $\Xi$ be a subset of $\mathbb{N}$. Let $\mathcal{D}_i\subseteq\mathcal{E}$ with $i\in\Xi$ be non-empty disjoint subsets covering $\mathcal{E}$: $\mathcal{D}_i \cap \mathcal{D}_j = \emptyset$ if $i\neq j$ and $\underset{i\in\Xi}{\cup} \mathcal{D}_i=\mathcal{E}$. Let $\mathcal{Q}=\{q_i\}_{i\in \Xi}$ such that $q_i\in \mathcal{D}_i$ for any $i\in \Theta$. {A function $\mathfrak{q}:\mathcal{E}\mapsto \mathcal{Q}$ is said to be a quantization function (\textit{quantizer}) on the set $\mathcal{E}$ if $\mathfrak{q}(x)=q_i$ for all $x\in \mathcal{D}_i$ and all $i\in \Xi$. }

The set $\mathcal{Q}$ represents a collection of all quantization values.
Each disjoint set $\mathcal{D}_i$ is called by a quantization cell, and the vector $\mathfrak{q}_i$ is the quantization seed of $\mathcal{D}_i$. 
In this paper, we assume that the whole state vector is available but it is quantized such that the set $\mathcal{Q}$ is finite and time-invariant. This case corresponds to the so-called \textit{finite static quantizer}.
{We refer the reader to  \cite{brockett_Liberzon2000TAC}, \cite{bullo_Liberzon2006TAC} for more details about state quantization.}

The feedback control  with quantization is given by: \vspace{-1mm}
\begin{equation}\label{eq:f_q}
        u = \tilde{u}(\mathfrak{q}(x)), \vspace{-1mm}
    \end{equation}
    where $\tilde{u}:\mathbb{R}^n\mapsto \mathbb{R}^m$ is a static feedback law, $\mathfrak{q}(x)$ is a quantized state.
    The closed-loop system has discontinuous
right-hand side. Its solutions are understood
in the sense of Filippov (see \cite{filippov2013differential}).

 The aim is to address the problem of finite-time stabilization  by a static feedback with \textit{finite static quantizer}. 
In other words, we need to design a quantizer  $\mathfrak{q}:\mathcal{E}=\mathbb{R}^n\mapsto \mathcal{Q}\subset\mathbb{R}^n, |\mathcal{Q}|<+\infty$ and a feedback law $\tilde u:\R^n\mapsto \R^n$ such that $\mathfrak{q}$ is time-invariant and the closed-loop system \eqref{eq:lin_sys}, \eqref{eq:f_q} is globally or locally uniformly finite-time stable
\footnote{The system \eqref{eq:lin_sys} is said to be globally (locally) uniformly \vspace{-2mm}
\begin{itemize}
\item 
\textit{Lyapunov stable} if $\exists \alpha\in \mathcal{K}$ such that $\|x(t)\|\leq \alpha(\|x_0\|)$,  $\forall t\geq 0$, 
$\forall x_0\in \R^n$ (resp., $\in \Omega$ - a neighbourhod of $\mathbf{0}$);\vspace{-2mm}
\item \textit{finite-time stable} (see, e.g., \cite{Roxin1966:RMCP}, \cite{bhat2000finite}, \cite{Orlov2004:SIAM}) if it is globally (resp., locally) uniformly Lyapunov stable and there exists a locally bounded function $T: \R^n \mapsto \R_+$ such that $\|x(t)\|=0, \forall t\geq T(x_0), \forall x_0\in \R^n$ (resp., $\in \Omega$).\vspace{-2mm}
\end{itemize}
}.
More specifically, we consider the problem of designing a finite static quantizer for a given finite-time stabilizer $\tilde{u}$. We are also interested in finding the size of the data, which has to be transmitted through the network, in order to use the feedback $\tilde{u}(\mathfrak{q}(x))$. This size is deremined by the number of bits $M$ such that 
$|\mathcal{Q}| \leq \log_2{M}$.
We show that the considered problems can be solved for the so-called generalized homogeneous stabilizers (see, e.g., \cite{polyakov2016IJRNC}, \cite{polyakov2019IJRNC}).

\section{Preliminary}
\label{sec:pre}

\subsection{Homogeneity and homogeneous function, system}
Homogeneity refers to a class of dilation symmetries, which have been shown to possess several useful properties for control design and analysis \cite{zubov1958systems}, \cite{khomenyuk1961systems}, \cite{kawski1991}, \cite{rosier1992SCL},  \cite{bhat_etal_2005_MCSS}. 
\begin{definition}[\cite{kawski1991}]
	\textit{A mapping $\dn(s):\mathbb{R}^n\mapsto \mathbb{R}^n$, $s\in\mathbb{R}$ is said to be a  dilation in $\mathbb{R}^n$ if }
	\begin{itemize}
		\item \textit{$\dn(0) = I_n$, $\dn(s+t) = \dn(s)\dn(t)=\dn(t)\dn(s)$, $\forall s, t\in\mathbb{R}$;}
		\item  \textit{$\lim\limits_{s\rightarrow -\infty}\|\dn(s)x\| =0$ and $\lim\limits_{s\rightarrow \infty}\|\dn(s)x\| =\infty$, $\forall x\neq \mathbf{0}$.}  
	\end{itemize}
\end{definition}
The dilation $\dn$ is continuous (linear) if the mapping $s\mapsto \dn(s)$ is continuous (resp., linear). 
Any \textit{linear} continuous dilation in $\R^n$ is given by \vspace{-2mm}
\begin{equation}\label{eq:dilation}
\dn(s) = e^{sG_\dn}:=\textstyle\sum_{i=0}^{\infty} \tfrac{s^i G_\dn^i}{i!}, \vspace{-2mm}
\end{equation}
where an anti-Hurwitz matrix $G_\dn$ is  known as the \textit{generator} of the dilation. For  $G_\dn = I_n$ we have the standard dilation
$e^sI_n$, while a diagonal $G_\dn$ corresponds to the weighted dilation (see  \cite{zubov1958systems}, \cite{rosier1992SCL}).

\begin{definition}[\cite{polyakov2019IJRNC}]\label{def:monocity}
A dilation $\dn$ is strictly monotone with respect to a norm $\|\cdot\|$ in $\R^n$ 
if $\exists \beta\!>\!0$ :
	$
	\|\mathbf{d}(s)\| \!\leq\! e^{\beta s}, \forall s \!\leq\! 0 
	$.
\end{definition}
We deal only with the linear continuous dilation given by \eqref{eq:dilation}. The subsequent condition is a direct result of the quadratic Lyapunov function theorem for linear systems.
\begin{proposition}
\label{prop: 1}
	\textit{A linear continuous dilation $\mathbf{d}$ in $\R^n$ is strictly monotone  with respect to the norm $\|z\|=\sqrt{z^\top P z}$ if and only if the following linear matrix inequality holds \vspace{-2mm}
	\begin{equation} P\succ 0, \
	P G_\dn+G_\dn^{\top} P \succ 0,\label{eq:mon_cond} \vspace{-2mm}
	\end{equation}
	where $G_\dn \in \mathbb{R}^{n\times n}$ is the generator of $\mathbf{d}$. }
\end{proposition}

Inspired by \cite{kawski1991} the homogeneous function and vector field are defined  as follows.
\begin{definition}[\cite{kawski1991}]
	A  vector field $f:\R^n\mapsto \R^n$ (a function $h:\R^n \!\mapsto\! \R$) is said to be $\dn$-homogeneous of degree $\mu\!\in\! \R$ if \vspace{-2mm}
	\[
 \begin{array}{c}
	f(\dn(s))=e^{\mu s} \dn(s) f(x), \quad \forall x\in\R^n, \quad \forall s\in \R,\\
	(\text{resp., } h(\dn(s))=e^{\mu s} h(x), \quad \forall x\in\R^n, \quad \forall s\in \R),
 \end{array} \vspace{-2mm}
	\]
	where $\dn$ is a continuous dilation in $\R^n$.
\end{definition}

The linear continuous dilation $\dn$ introduces an alternative topology in $\mathbb{R}^n$ by the so-called "homogeneous norm" \cite{kawski1995IFAC}, \cite{grune2000SIAM} providing related geometric objects (e.g., sphere).
\begin{definition}[\cite{polyakov2020book}]
    The set \vspace{-2mm}
    \[
    S_\dn^{n-1}(r) = \{x\in\mathbb{R}^n:\|\dn(-\ln r)x\|=1\}, \vspace{-2mm}
    \]
    is called a $\dn$-homogeneous sphere of the radius $r>0$ with the center at $\boldsymbol{0}$.
\end{definition}
The unit homogeneous shpere $S_\dn^{n-1}(1)$ coincides with the unit sphere $S^{n-1}(1)$ in Euclidean space and \vspace{-2mm}
\begin{equation}
    S_\dn^{n-1}(r) = \dn(-\ln r)S_\dn^{n-1}(1) = \dn(-\ln r )S^{n-1}(1). \vspace{-2mm}
\end{equation}
The above equation means that any homogeneous sphere can be obtained through scaling the unit sphere.

\begin{definition}[\cite{polyakov2019IJRNC}]\label{def:hom_norm}
	\textit{Let a continuous linear dilation $\mathbf{d}$ be strictly  monotone  with respect to a norm $\|\cdot\|$ in $\R^n$. The function $\|\cdot\|_\dn: \mathbb{R}^n \mapsto\R_+$ defined as $\|\mathbf{0}\|_{\dn}=0$ and 
	$$
	\|x\|_{\mathbf{d}}=e^{s}, \text { where } s \in \mathbb{R}:\left\|\mathbf{d}\left(-s\right) x\right\|=1, \quad x\neq \mathbf{0},
	$$
	is called the canonical homogeneous norm in $\mathbb{R}^n$.}
\end{definition}

The canonical homogeneous is a continuous function defined implicitly and its derivative has the form (see \cite{polyakov2020book}) \vspace{-2mm}
\begin{equation}  \!\tfrac{\partial\|x\|_\dn}{\partial x} \!=\! \|x\|_\dn\tfrac{x^\top\dn^\top(-\ln\|x\|_\dn)P\dn(-\ln\|x\|_\dn)}{x^\top\dn^\top(-\ln\|x\|_\dn)PG_\dn\dn(-\ln\|x\|_\dn)x}. \vspace{-2mm}
\end{equation}
provided that $\|x\|_{\dn}$ is induced by the weighted Euclidean norm $\|x\|=\sqrt{x^{\top}Px}$ with $P$ satisfying \eqref{eq:mon_cond}.
 
 The  homogeneity of for a linear vector field is characterized  as follows (\cite{polyakov2019IJRNC}, \cite{zimenko_etal_2020_TAC}).
\begin{lemma}\label{lem:nilpotent}
Let $\dn$ be a linear continuous dilation. The following claims are equivalent:\vspace{-2mm}
\begin{itemize}
    \item the linear vector field $x\mapsto Ax$ with $x\in\mathbb{R}^n$ and $A\in\mathbb{R}^{n\times n}$ is $\dn$-homogeneous of degree $\mu\neq 0$;
    \item the identity holds\vspace{-2mm}
\begin{equation}\label{eq:A}
AG_\dn = (\mu I + G_\dn)A, \vspace{-2mm}
\end{equation}
where $G_\dn\in\mathbb{R}^{n\times n}$ is a generator of linear dilation $\dn$;
\item the matrix $A$ is nilpotent;
\end{itemize}
\end{lemma}

An asymptotically stable homogeneous system exhibits finite-time stability in the case on negative degree \cite{bhat_etal_2005_MCSS}. The following  homogeneous stabilizer for LTI system has been developed the series of the works \cite{polyakov2019IJRNC}, \cite{zimenko_etal_2020_TAC} and \cite{polyakov2020book}.

\begin{theorem} \label{thm:hom_linear}
A controllable linear system \eqref{eq:lin_sys} with the control law	\vspace{-2mm}\begin{equation}\label{eq:hom_con_P}
		u(x)\!=\!K_{0} x\!+\!\|x\|_{\dn}^{1+\mu}K \dn \left(-\ln \|x\|_{\dn}\right) x,\ \mu\!\in\![-1,0), \vspace{-2mm} 
	\end{equation}
 is globally uniformly  finite-time stable, where  $K = YX^{-1}$, $G_{\dn}=I_n+\mu G_0$, $K_{0}=Y_0(G_0-I_n)^{-1}$,  the pair $G_0\in \R^{n\times n}$, $Y_0\in \R^{m\times n}$ is a solution of the matrix equation:\vspace{-2mm}
 \begin{equation}\label{eq:G0}
		AG_0+BY_0=G_0A+A, \quad G_0B=\boldsymbol{0}. \vspace{-2mm}
\end{equation}
and the pair $X\in\R^n$ and $Y\in \R^{m\times n}$ is a solution of the linear algebraic system:\vspace{-2mm}
\begin{equation}
		\!\!\!\!\left\{\!\begin{aligned}
			&\!X A_{0}^{\top}\!+\!A_{0} X\!+\!Y^{\top}\! B^{\top}\!\!+\!B Y\!+\!\rho\!\left(X G_{\dn}^{\top}\!+\!G_{\dn} X\right)\!=\!\mathbf{0}, \\
			&\!X G_{\dn}^{\top}+G_{\dn} X \succ 0, \quad X \succ 0.
		\end{aligned}\right. \vspace{-1mm}
		\label{eq:LMI0}
	\end{equation}
\end{theorem}

\begin{itemize}
  \item If the pair $\{A,B\}$ is controllable then the linear algebraic systems \eqref{eq:G0} and \eqref{eq:LMI0} always have solutions (see \cite{polyakov2016IJRNC}).
        \item Canonical homogeneous norm $\|x\|_\dn$ is a Lyapunov function for the closed system \eqref{eq:lin_sys}, \eqref{eq:hom_con_P} and 
    \vspace{-2mm}\begin{equation}\label{eq:d_V_x}
            \tfrac{\partial \|x\|_\dn}{\partial x}(Ax+Bu(x))= -\rho \|x\|_\dn^{1+\mu}. \vspace{-2mm}
        \end{equation}
        \item The linear feedback $K_0x$ homogenizes the system and guarantees that the matrix $A_0=A+BK_0$ is nilpotent. 
        \item The closed-loop system with the control \eqref{eq:hom_con_P} 
        is $\dn$-homogeneous of degree $\mu$.
    \end{itemize}
For $\mu=-1$, solutions of the closed-loop system \eqref{eq:f_q} are understood in the sense of Filippov. Let us recall the Lyapunov theorem
for differential equations with discontinuous right-hand sides. 
\begin{theorem}[\cite{filippov2013differential},  page 152]\label{thm:Filippov}
    Consider the system
$
\dot{x} = f(t, x)$, $ t>0
$,
where $x$ is the state vector and $f$ is a piece-wise continuous vector field. Suppose there exists a positive definite radially unbounded function $V:C(\mathbb{R}^n, \mathbb{R}_+)\cap C^{1}(\mathbb{R}^n\setminus\{\boldsymbol{0}\}, \mathbb{R}_+)$:
 $$
 \dot{V}(x)=\tfrac{\partial V}{\partial x} f(t,x)< 0,
 $$ almost everywhere in $\R\times\R^n$.
Then, the system is globally asymptotically stable provided that its solutions are understood in the sense of Filippov.
\end{theorem}
The above result means that the stability analysis of the discontinuous ODE can be done in the conventional way (using the right-hand side of the ODE)  provided that Lyapunov function $V$ is continuously differentiable away from the origin and solutions are understood in the sense of Filippov. A Lipschitz continuous function $V$ cannot be utilized in the same way (see, e.g., \cite{utkin2013sliding}) despite its differentiability almost everywhere.

\section{Main results}\label{sec:main}

\subsection{Generalized homogeneous spherical quantizer}

 One of key features of homogeneous systems is a completeness of its global definition based on local data. Indeed,
 if a homogeneous vector field is defined on a sphere then it can be expanded to the whole $\mathbb{R}^n\backslash\{\boldsymbol{0}\}$ using the dilation (see the Definition 4). This inspires the spherical quantization for homogeneous systems. 

 It is known that any non-zero element of $\mathbb{R}^n$ can be projected to the unit sphere using the uniform dilation. The corresponding projector $\pi: \mathbb{R}^n\setminus\{\boldsymbol{0}\} \mapsto S^{n-1}(1)$ is defined as \vspace{-1mm}
$$
\pi(x)=\tfrac{x}{\|x\|}=\mathbf{d}(-\ln \|x\|) x, \vspace{-1mm}
$$
where $\mathbf{d}(s)=e^{s I_n}$ is the uniform dilation. A similar projection can be constructed using any strictly monotone continuous dilation $\mathbf{d}$ (see \cite{polyakov2020book}).
Indeed, for a given strictly monotone dilation $\dn(s) = e^{G_\dn s}$, $s\in\mathbb{R}$, the canonical homogeneous norm $\|x\|_\dn$ is continuous on $\mathbb{R}^n$ and Lipschitz continuous on $\mathbb{R}^n\setminus \{\boldsymbol{0}\}$.
Moreover, the canonical homogeneous norm introduces a $\dn$-homogeneous projector $\pi_\dn:\mathbb{R}^n\setminus\{\boldsymbol{0}\}\mapsto S^{n-1}(1)$,\vspace{-1mm}
\begin{equation}
    \pi_\dn(x) = \dn(-\ln\|x\|_\dn)x.\vspace{-1mm}
\end{equation}
since, by definition of the canonical homogeneous norm we have $\|\pi_\dn(x)\|=1$. Notice that    $\pi_\dn(\dn(s)x) = \pi_\dn(x)$, $\forall s\in\mathbb{R}, \forall x\neq \boldsymbol{0}$. 

\begin{definition}
Let a linear continuous dilation $\dn$ be strictly monotone  with respect to the weighted Euclidean norm  $\|x\|=\sqrt{x^\top P x}$, $x\in\mathbb{R}^n$, and $\mathcal{Q}_s\subset S^{n-1}(1)$ be a discrete set. A quantizer $\mathfrak{q}_{\pi_\dn}: \mathbb{R}^n\mapsto \mathcal{Q}_s$  given by \vspace{-1mm}
\begin{equation}
    \mathfrak{q}_{\pi_\dn}(x) = \left\{
        \begin{aligned}
         &\mathfrak{q}_s(\pi_\dn(x)), &x\neq \boldsymbol{0},\\
         & \boldsymbol{0}, & x= \boldsymbol{0},
        \end{aligned}
    \right. \vspace{-1mm}
\end{equation}
is called $\dn$-homogeneous spherical quantizer, 
where $\mathfrak{q}_s:S^{n-1}(1)\mapsto \mathcal{Q}_s\subset S^{n-1}(1)$ is a discrete map, $|\mathcal{Q}_s|<+\infty$.
\end{definition}
 
 The quantizer $\mathfrak{q}_{\pi_\dn}$ given above is $\dn$-homogeneous (symmetric with respect to dilation $\dn$) as follows \vspace{-1mm}
\[
\mathfrak{q}_{\pi_\dn}(\dn(s)x) = \mathfrak{q}_{\pi_\dn}(x), \ \forall s\in\mathbb{R}, \ \forall x\in\mathbb{R}^n. \vspace{-1mm}
\]
The homogeneous spherical quantization consists of two steps:  first, state $x\in\mathbb{R}^n\setminus\{\boldsymbol{0}\}$ in projected to the unit sphere using the homogeneous projector $\pi_\dn$, and, next, the discrete map $\mathfrak{q}_s$ quantizes the projected state.
Since the value $\pi_\dn(x)$ belongs to the unit sphere, the following proposition is straightforward due to the compactness of the sphere.
\begin{proposition}\label{prop:eps}
    For a $\dn$-homogeneous spherical quantizer $\mathfrak{q}_{\pi_\dn}:\mathbb{R}^n\mapsto \mathcal{Q}_s\subset S^{n-1}(1)$,  any given $\delta>0$, there exists a set $\mathcal{Q}_s$ such that $|\mathcal{Q}_s|=N<+\infty$ and \vspace{-2mm}
    \[
    \|\mathfrak{q}_{\pi_\dn}(z)-z\|\le \delta, \quad \forall z\in S^{n-1}(1).
    \]
\end{proposition}
\vspace{-5mm}
The $\dn$-homogeneous spherical quantizer is a finite static quantizer being invariant with respect to linear dilation $\dn$. Issues of practical realization of the spherical quantization algorithm  are discussed in Section \ref{sec:relization}.

\subsection{Finite-time stabilization with homogeneous spherical quantization}

The dilation symmetry of a closed-loop system  may not be preserved in the presence of quantization. For example, the scalar closed-loop system \vspace{-1mm}
\[
\dot{x} = ax \!+\! u(x),\ a\!\in\!\mathbb{R},  \ u(x) \!=\! -|x|^{-0.5}sign(x) -ax. \vspace{-1mm}
\]
is standard homogeneous of  degree $-0.5$. However, for $a\neq0$,  the open-loop system and state feedback  separately do not admit the same dilation symmetry. That is why the closed-loop system with  the quantized state feedback $\dot x=ax + u(\mathfrak{q})$ is not homogeneous in the general case.  
Based on this observation we first study the case, when the $\dn$-homogeneous spherical quantization does not destroy the dilation symmetry of the closed-loop system. For MIMO LTI system, the latter is possible if the matrix $A$ is $\dn$-homogeneous (see Lemma \ref{lem:nilpotent}). In this case, the equation \eqref{eq:G0} has solution $Y_0 = K_0 = \boldsymbol{0}$, the feedback \eqref{eq:hom_con_P} contains only the nonlinear term: \vspace{-1mm}
\[
u_{\text{hom}}(x) = \|x\|_\dn^{1+\mu}K\dn(-\ln\|x\|_\dn)x. \vspace{-1mm}
\] 
The open loop and closed loop systems are homogeneous, and the homogeneity property allows us to project $u_{\text{hom}}$ onto the $\dn$-homogeneous sphere. 
Taking into account$\|\pi_{\dn}(x)\|_\dn=1$, we derive:\vspace{-2mm}
\[
u_{\text{hom}}(\pi_{\dn}(x)) = K\pi_{\dn}(x), \vspace{-2mm}
\]
where $\pi_{\dn}$ is the $\dn$-homogeneous projector (see above).
The homogeneous control becomes a static feedback of the projected state. Following this projection strategy,  homogeneous feedback with a $\dn$-homogeneous spherical quantizer is constructed, and stability is investigated. 

\begin{theorem}\label{thm:limited_data}
  Let the parameters  $G_0\in \R^{n\times n}$, $K_0\in \R^{ m\times n}$, $G_{\dn}=I_n-G_0$ and $A_0=A+BK_0$ be defined as in Theorem \ref{thm:hom_linear}. 
Let  $X \in \mathbb{R}^{n \times n}$ and $Y \in \mathbb{R}^{m \times n}$ be  a solution of the LMI \vspace{-1mm}
\begin{subequations}\label{eq:LMI_1}    
    \begin{align}
        &\quad
        XG_\dn^\top + G_\dn X\succ 0, \ X\succ 0,\label{eq:LMI_1_1}\\    &\begin{bmatrix}
X A^{\top}_0 \!+\! A_0 X \!+\! Y^{\top}\! B^{\top} \!+\! B Y \!+\! \delta^2 \tau X & BY \\
Y^\top B^\top & -\tau X
\end{bmatrix} \!\!\prec\! 0. \label{eq:LMI_1_3}
\end{align} \vspace{-1mm}
\end{subequations}
for some $\delta>0$ and some $\tau>0$.
Let the canonical homogeneous norm $\|\cdot\|_{\dn}$ be induced by norm $\|x\|=\sqrt{x^{\top}X^{-1}x}$. 
Let $\mathfrak{q}_{\pi_\dn}:\mathbb{R}^n\mapsto \mathcal{Q}_s\subset S^{n-1}(1)$ be a $\dn$-homogeneous spherical quantizer such that \vspace{-1mm}
$$
\|\mathfrak{q}_{\pi_\dn}(z)-z\|\le \delta, \ \forall z\in S^{n-1}(1). \vspace{-1mm}
$$
Then, the closed-loop system \eqref{eq:lin_sys}, 
\begin{equation}\label{eq:u_c_2}
u =  K\mathfrak{q}_{\pi_\dn}(x),\quad   K = YX^{-1}. 
\end{equation}
 with $g=\boldsymbol{0}$ is  globally (locally) uniformly finite-time stable if $K_0=\mathbf{0}$  (resp.,  $K_0\neq\mathbf{0}$).
\end{theorem}
\begin{proof}
If $K_0=\mathbf{0}$ then $A_0=A$ and the LTI system in $\dn$-homogeneous: \vspace{-1mm}
\begin{equation}\label{eq:A_B_d}
    \dn(s)A_0 = e^{- s}A_0\dn(s), \ \dn(s)B = e^{s}B, \ s\in \mathbb{R}. \vspace{-1mm}
\end{equation}
where $\dn$ is generated by $G_{\dn}=I_n-G_0$. In this case, the closed-loop system is $\dn$-homogeneous due to $\dn$-homogeneity of the quantizer and $\dn$-homogeneity of the controller.
Calculating the time derivative of the Lyapunov function candidate $\|x\|_\dn$ along trajectories of the system \eqref{eq:f_q}, \eqref{eq:u_c_2} we derive \vspace{-1mm}
  \begin{equation}
    \tfrac{d \|x\|_\dn}{dt} \!= \!\|x\|_\dn\tfrac{2x^\top\dn^\top(-s_x)P\dn(-s_x)[A_0x + BK\mathfrak{q}_{\pi_\dn}\left(x\right)]}{x^\top\dn^\top(-s_x)(G_\dn^\top P + PG_\dn)\dn(-s_x)x}, \vspace{-1mm}
\end{equation}
where $s_x=\ln\|x\|_\dn$.
  
Using the dilation invariance of $\mathfrak{q}_{\pi_\dn}$, we obtain   \vspace{-1mm}
\begin{equation}\label{eq:d_V}
        \tfrac{d \|x\|_\dn}{dt} \!=\! 
    \tfrac{\left[\begin{smallmatrix}\pi_\dn(x)\\
            \sigma
\end{smallmatrix}\right]^\top\!\!\left[\begin{smallmatrix}    
(A_0\!+\!BK)^\top\! P + P(A_0 \!+\! BK) & PBK\\
K^\top B^\top P & \boldsymbol{0}
\end{smallmatrix}\right]\!\!
\left[\begin{smallmatrix}\pi_\dn(x)\\
            \sigma
\end{smallmatrix}\right]}{\pi_\dn^\top(x)(G_\dn^\top P + PG_\dn)\pi_\dn(x)},
\end{equation}
where $\sigma=\mathfrak{q}_{\pi_\dn}(\pi_{\dn}(x))-\pi_{\dn}(x)$.
On the other hand, since $\|\sigma\|\le\delta $ and  $ \|\pi_\dn(x)\|=1$,  it holds \vspace{-1mm}
\begin{equation}
\|\sigma\|\!\le\!\delta\|\pi_\dn(x)\|\,\Leftrightarrow\,\left[\begin{smallmatrix}\pi_\dn(x)\\
            \sigma
\end{smallmatrix}\right]^{\!\top}\!\!\left[\begin{smallmatrix}    
 -\delta^2 P & \boldsymbol{0}\\
\boldsymbol{0} & P
\end{smallmatrix}\right]\!\!
\left[\begin{smallmatrix}\pi_\dn(x)\\
            \sigma
\end{smallmatrix}\right]\!\le\! 0. \vspace{-1mm}
\end{equation}
Then according to S-procedure \cite{boyd2004convex}, the norm $\|x\|_\dn$ is Lyapunov function of system \eqref{eq:lin_sys}, \eqref{eq:u_c_2} if and only if there exists a positive constant $\tau>0$ such that \vspace{-1mm}
\begin{equation}
W:=\!\!\left[\begin{smallmatrix}
 A^{\top}_0P + PA_0 + K^{\top} B^{\top}P + B KP + \delta^2 \tau P & PBK \\
K^\top B^\top P & -\tau P
\end{smallmatrix}\right]\prec 0. \vspace{-1mm}
\end{equation}
This matrix inequality is equivalent to \eqref{eq:LMI_1_3} for $X=P^{-1}$ and $K = YX^{-1}$.

{Since $W\prec 0$, $P\succ 0$ and $G_\dn^\top P + P G_\dn\succ 0$ then there exists $\rho>0$ : $W\le -\rho \left(\begin{smallmatrix} G_\dn^\top P + P G_\dn & \mathbf{0}\\\mathbf{0} & P\end{smallmatrix} \right)$. Hence, taking into account $\pi_{\dn}^{\top}(x) P\pi_{\dn}(x)=1$ we derive $\tfrac{d\|x\|_\dn}{dt} < -\rho$  almost everywhere on $\R^n$.  Applying Theorem \ref{thm:Filippov} we conclude 
global asymptotic stability of the system, while the negative homogeneity degree of the closed-loop system implies the finite-time stability.
}
For $K_0\neq \mathbf{0}$, repeating the above considerations we derive \vspace{-1mm}
\[
\tfrac{d \|x\|_\dn}{dt} \leq -\rho- \tfrac{x^\top\dn^\top(-s_x)PBK_0x}{x^\top\dn^\top(-s_x)PG_\dn\dn(-s_x)x}\leq -\rho+\tfrac{\|BK_0x\|}{\beta}, \vspace{-1mm}
\]
where $\beta=\frac{1}{2}\lambda_{\min}\left(P^{\frac{1}{2}}G_{\dn}P^{-\frac{1}{2}}+P^{-\frac{1}{2}}G_{\dn}^{\top}P^{\frac{1}{2}}\right)>0$ and the identity $x^{\top} \mathbf{d}^{\top}\left(-\ln \|x\|_{\mathbf{d}}\right) P \mathbf{d}\left(-\ln \|x\|_{\mathbf{d}}\right) x=1$ is utilized in the derivations. Hence, we obtain the local finite-time stability.
\end{proof}
    The feasibility of LMIs \eqref{eq:LMI0} has been proven in \cite{polyakov2016IJRNC}. Taking $\tau = \tfrac{1}{\delta}$ we conclude the feasibility of 
     \eqref{eq:LMI_1}, at least, for small $\delta$. 
     
\begin{corollary} \label{cor:1}
  Let all conditions of Theorem \ref{thm:limited_data} hold. If there exists sufficiently small $\kappa>0$, such that for any $ x \in \mathbb{R}^n \backslash\{\mathbf{0}\}, t \geq 0$, 
  \vspace{-1mm}\begin{equation}\label{eq:pert}
  \|x\|_{\dn}\tfrac{x^{\top}\mathbf{d}^{\top}\left(-\ln \|x\|_{\mathbf{d}}\right) P\mathbf{d}\left(-\ln \|x\|_{\mathbf{d}}\right) g(t,x)}{x^{\top} \mathbf{d}^{\top}\left(-\ln \|x\|_{\mathbf{d}}\right) P G_{\mathbf{d}} \mathbf{d}\left(-\ln \|x\|_{\mathbf{d}}\right) x} \!\leq\! \kappa, \ \forall t\!\geq\! 0, \vspace{-1mm}
  \end{equation}
  almost everywhere on $\R^n$.
  Then, the perturbed closed-loop system \eqref{eq:lin_sys}, \eqref{eq:u_c_2} is globally (locally) uniformly finite-time stable if $K_0=\mathbf{0}$ (resp., $K_0\neq \mathbf{0}$). 
\end{corollary}
Indeed, 
as shown in the proofs of Theorem \ref{thm:limited_data}, there exits a positive constant $\rho$ such that $\tfrac{\partial\|x\|_\dn}{\partial x}(Ax+BK\mathfrak{q}_{\pi_\dn}(x))\le -\rho+\tfrac{\|BK_0x\|}{\beta}$.
    Then  for the perturbed system we have \vspace{-1mm}
    \begin{equation*}
        \begin{aligned}
            \tfrac{d\|x\|_\dn}{dt} &\!\le\! -\rho \!+\! \tfrac{\|BK_0x\|}{\beta}+\|x\|_\dn\tfrac{x^\top\dn^\top(-s_x)P\dn(-s_x)g(t,x)}{x^\top\dn^\top(-s_x)PG_\dn\dn(-s_x)x}\\
            & \le -(\rho-\kappa) +\tfrac{\|BK_0x\|}{\beta}.
        \end{aligned} \vspace{-1mm}
    \end{equation*}   
Notice that the inequality \eqref{eq:pert} is fulfilled, for example, if \vspace{-1mm}
\[
g(t,x)=B\gamma(t,x) :  \ \|B\gamma(t,x)\|\leq \beta\kappa, \;\forall t\geq 0,\; \forall x\in \R^n. \vspace{-1mm} 
\] Indeed, in this case, due to 
$\dn(s)B=e^sB$, we have \vspace{-1mm}
\[
\|x\|_{\dn}\tfrac{x^{\top}\mathbf{d}^{\top}\left(-\ln \|x\|_{\mathbf{d}}\right) P\mathbf{d}\left(-\ln \|x\|_{\mathbf{d}}\right) g(t,x)}{x^{\top} \mathbf{d}^{\top}\left(-\ln \|x\|_{\mathbf{d}}\right) P G_{\mathbf{d}} \mathbf{d}\left(-\ln \|x\|_{\mathbf{d}}\right) x}\leq \tfrac{\|B\gamma\|}{\beta}\leq \kappa. \vspace{-1mm}
\]
The simplest example of the homogeneous system \eqref{eq:lin_sys} is the controlled integrator chain:
$A=\left[\begin{smallmatrix} 0 & 1 & 0 & ...  & 0 & 0\\ 0 & 0 &1 & ...& 0 & 0\\
... & ... & ... & ... & ... & ...\\
0 & 0 & 0 & ... & 0 & 1\\
0 & 0 & 0 & ... & 0 & 0\\
\end{smallmatrix}\right], B=\left[\begin{smallmatrix} 0  \\ 0\\
... \\
0 \\
1\\
\end{smallmatrix}\right]$.

\subsection{A possible spherical quantizer design}\label{sec:relization}

To guarantee the feasibility of LMI \eqref{eq:LMI_1} the parameter $\delta$ has to be small enough.
This means that  quantization  on the unit sphere has to be sufficiently dense. 
However, in practice, the quantization number is often constrained by limited memory or limited data transmission (network bandwidth). For an implementation of the quantized feedback law designed in Theorem \ref{thm:limited_data}, an admissible number $M$ of bits for encoding (numbering) the quantization seeds  (i.e., $|\mathcal{Q}|\le N=2^{M}$) is given in practice.  In this case, an explicit formula relating the quantization error $\delta$  with the number of bits is required for the spherical quantizer design. In this section, given an integer $N$ and $P=X^{-1}$ satisfying  LMIs \eqref{eq:LMI_1}, we design a quantizer $\mathfrak{q}(x)$ to ensure $\|\mathfrak{q}(x) - x\| \leq \delta_N$, where $\delta_N>0$ is the quantization error for the given $N$. To provide an explicit formula for $\delta_N$, we design the homogeneous spherical quantizer based on spherical coordinates.

Any unit vector $z=[z_1,z_2,\cdots,z_n]^\top\in S^{n-1}(1)$ be represented using spherical coordinates: $\xi = [1, \varphi_1, \varphi_2, \cdots, \varphi_{n-1}]^\top$, $0\le \varphi_1, \varphi_2,\cdots, \varphi_{n-2}\le \pi$, $0\le \varphi_{n-1}<2\pi$.
Let the map from spherical coordinates to Cartesian coordinates be denoted by $g_1: \mathbb{R}_+\times[0,\pi]^{n-1}\times [0,2\pi)\mapsto S^{n-1}(1)$, \vspace{-2mm}
$$
z \!=\! g_1(\xi), z_1 \!=\!\cos\varphi_1, z_n\!=\!\!\prod_{j=1}^{n-1} \!\sin\varphi_j, z_k \!=\! \cos\varphi_k \!\prod_{j=1}^{k-1} \!\sin\varphi_j, \vspace{-2mm}
$$
where
$ k=2, 3 \cdots, n-1$.
Let the inverse map (from Cartesian coordinates to  spherical coordinates) be denoted by $g_2: S^{n-1}(1)\mapsto \mathbb{R}_+\times[0,\pi]^{n-1}\times[0,2\pi)$, \vspace{-2mm}
$$
\xi = g_2(z), \varphi_i = \operatorname{atan2}(\sqrt{\scriptstyle\sum_{j=i+1}^{n}x_j^2},x_i) , \ i = 1, \ldots, n-2, \vspace{-1mm}
$$
\vspace{-2mm}
$$
\varphi_{n-1} =\operatorname{atan2}(x_n,x_{n-1}),
$$
where $\operatorname{atan2}(\cdot,\cdot)$ is the so-called two-argument arctangent function.

\begin{proposition}\cite{wang2021TAC}\label{prop:Delta}
Let $z$, $y\in S^{n-1}(1)$ be two unit vectors with spherical coordinates
$\xi_z \!=\!g_2(z)$ and 
$
\xi_y \!=\! g_2(y),$ respectively.
If for any $i\in 1, 2, \cdots, n-1$, it holds that $|\varphi_{z,i} - \varphi_{y,i}|\le \Delta$, $\Delta\in(0,\tfrac{\pi}{2})$
 then \vspace{-2mm}
    \begin{equation}
        |z-y|\le \sqrt{2-2\left(2\cos^{2(n-1)}(\tfrac{\Delta}{2})-1\right)}. \vspace{-2mm} 
    \end{equation} 
\end{proposition}
For a strictly monotone dilation $\dn$ with respect to $P\succ 0$,  the $\dn$-homogeneous spherical quantizer  in $\mathbb{R}^n$ is desfined as follows:\vspace{-1mm}
\begin{equation}\label{eq:partition}
\!\!\!\mathfrak{q}_{\pi_\dn}(x)\! =\! P^{-1/2}g_1([1, q(\varphi_{\pi_\dn,1}), \cdots, q(\varphi_{\pi_\dn,n-1})]^\top), \vspace{-1mm} 
\end{equation}
where
$
[1, \varphi_{\pi_\dn,1}, \varphi_{\pi_\dn,2}, \cdots, \varphi_{\pi_\dn,n-1}]^\top = g_2(P^{\tfrac{1}{2}}\pi_\dn(x)) 
$,
\[
q(\varphi_{\pi_\dn,k})\!=\!\left(\lfloor\tfrac{\varphi_{\pi_\dn,k}}{\Delta} \rfloor\!+\! \tfrac{1}{2}\right) \Delta, \quad \Delta\!\in(0,\tfrac{\pi}{2}), 
\]
$k \!=\! 1,2,\cdots, n-1$.

The proposed quantizer uses the uniform quantization of the angles of the spherical coordinates with equal intervals.
\begin{corollary}
\label{lem:q_err}
Let the $\dn$-homogeneous spherical quantizer  in $\mathbb{R}^n$ be defined by the formula \eqref{eq:partition} with $\Delta\in(0,\tfrac{\pi}{2})$ given by $\Delta\!=\tfrac{\pi}{\left\lfloor(N/2)^{1/(n-1)}\right\rfloor}$. Then the set of quantization seeds admits the estimate $|\mathcal{Q}_s|\le N$
    and the quantization error of \eqref{eq:partition} is bounded by $\|\mathfrak{q}_{\pi_\dn}(\pi_\dn(x)) \!-\! \pi_\dn(x)\| \!\le \delta_N$, \vspace{-2mm}
    \[
    \delta_N :=2\left[1 - \cos^{2(n-1)}\left(\tfrac{\pi}{2\left\lfloor(N/2)^{1/(n-1)}\right\rfloor}\right)\right]^{\frac{1}{2}}. \vspace{-2mm}
    \]
\end{corollary}
\begin{proof}
Indeed, 
according to Proposition \ref{prop:Delta}, we have \vspace{-2mm}
\begin{equation*}
    \begin{aligned}
        &
        \left|g_1([1, q(\varphi_{\pi_\dn,1}), \cdots, q(\varphi_{\pi_\dn,n-1})]^\top)-P^{\tfrac{1}{2}}\pi_\dn(x)\right|\\
        &\le
        2\left[1 - \cos^{2(n-1)}\left(\!\tfrac{\Delta}{2}\right)\right]^{\frac{1}{2}}.
    \end{aligned}
    \vspace{-2mm}
\end{equation*}
Hence,  $\|\mathfrak{q}_{\pi_\dn}(x) - \pi_\dn(x)\|\le 2\left[1 - \cos^{2(n-1)}\left(\!\tfrac{\Delta}{2}\right)\right]^{\frac{1}{2}}$. 
For each angle $\varphi_{\pi_\dn,i}$, $i = 1,2,\ldots, n-2$ the interval $[0,\pi/2]$ is divided into $\lfloor (\tfrac{N}{2})^{\tfrac{1}{n-1}}\rfloor$ non-overlapping intervals. Additionally, $\varphi_{\pi_\dn,n-1}$ is divided into $2\lfloor (\tfrac{N}{2})^{\tfrac{1}{n-1}}\rfloor$ non-overlapping intervals.
The number of quantization seeds is given by $|\mathcal{Q}_s| = 2\lfloor (N/2)^{\tfrac{1}{n-1}}\rfloor^{n-1}\leq N$. 
\end{proof}

    Geometrically, the uniform quantizer and logarithmic quantizer partition the space into disjoint rectilinear quantization cells. In contrast, the proposed homogeneous spherical quantizer divides the space into disjoint homogeneous cones\footnote{A nonempty set $\mathcal{H} \subseteq \mathbb{R}^n$ is said to be $\dn$-homogeneous cone in $\mathbb{R}^n$ (\cite{polyakov2020book})  if $\mathcal{H}$ is invariant with respect to the dilation $\dn$, i.e.,
$
\dn(s) z \in \mathcal{H} \text { for } z \in \mathcal{H} \text { and } s \in \mathbb{R}.
$} as quantization cells, where all cells are unbounded, and the quantization seeds lie on the surface of a sphere. An illustrative example for a two-dimensional space  is presented in Figure \ref{fig:L_U_S}.

\begin{figure}[htbp]
    \centering
    \includegraphics[width = 0.35\textwidth]{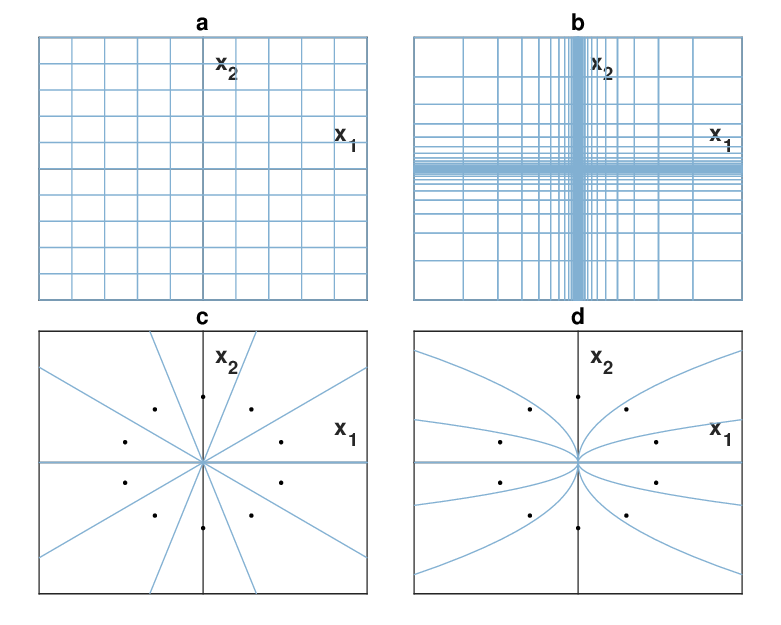}
    \caption{ Illustration of quantization in $\mathbb{R}^2$. (a) Uniform quantizer per axis. (b) Logarithmic quantizer = per axis. (c) Homogeneous spherical quantizer. 
    ($G_\dn=I_2$). (d) Homogeneous spherical quantizer ($G_\dn=\left[\begin{smallmatrix}2 & 0\\ 0 & 1\end{smallmatrix}\right]$). 
    The dots in (c) and (d) represent quantization  seeds on the sphere.
    }
    \label{fig:L_U_S}
\end{figure}

\section{Numerical example}\label{sec:sim}
For numerical validation, we consider the linear system with 
$
A = \left[\begin{smallmatrix}
    0  &  2  &  3\\
     0  &   0  &  4\\
     0   &  0  &  0
\end{smallmatrix}\right]$, $ B = \left[\begin{smallmatrix}
      0 \\
     0 \\
     1.5
\end{smallmatrix}\right]$.
The perturbation is considered as a function $g(x,t) = B(0.2\sin(t))$. The initial state is 
 $x_0 = [2,1,1]^\top$. 
 Since $A$ is nilpotent, then by solving equation \eqref{eq:G0}, we obtain $K_0=\boldsymbol{0}$, $G_\dn=\left[
 \begin{smallmatrix}
     3 & -0.75 & 0\\
     0 & 2 & \\
     0 & 0 & 1
 \end{smallmatrix}
 \right]$. 
 By solving the \eqref{eq:LMI_1} with $\delta=0.4$, $\tau = \tfrac{1}{\delta}$, we obtain \vspace{-2mm}
\[ P = \left[
\begin{smallmatrix}
    0.0053  &  0.0037  &  0.0185\\
    0.0037  &  0.0212  &  0.0381\\
    0.0185  &  0.0381  &  0.2522
\end{smallmatrix}
\right], \ 
K = -\left[
\begin{smallmatrix}
   0.1327 &  0.4089  & 1.7270
\end{smallmatrix}
\right]. 
\vspace{-2mm}
\]
 Following to Lemma \ref{lem:q_err}, we calculate the numbers of bits $N=8$ required to encode the quantized states:\vspace{-2mm}
\[
\|\mathfrak{q}_{\pi_\dn}(\pi_\dn(x))\!-\!\pi_\dn(x)\|\le \delta_N = 0.3896<\delta = 0.4. \vspace{-2mm}
\]
The homogeneous spherical quantizer is designed by the formula
\eqref{eq:partition}.
The simulation results are presented in Figures \ref{fig:sphere}, \ref{fig:x}, and \ref{fig:q}. Figure \ref{fig:sphere} illustrates the quantization algorithm with quantization seeds positioned on the weighted sphere. The evolution of the state of the closed-loop system is shown in Figure \ref{fig:x}. It validates the finite-time convergence of the state to the origin (up to discretization error) under perturbation. Figure \ref{fig:q} presents the quantization values during stabilization. The chattering phenomenon in Figure \ref{fig:q} results from the discontinuity in the projection $\pi_\delta(x)$ at $x = \mathbf{0}$.   

\begin{figure}[ht]
    \centering
\includegraphics[width = 0.35\textwidth]{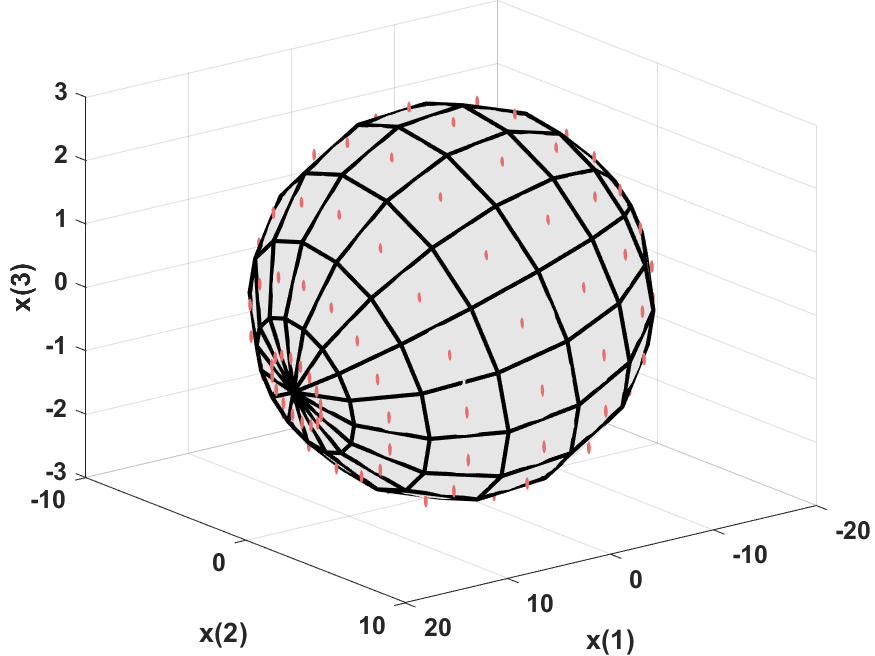}
    \caption{Quantization seeds (red markers)  with $N = 512$ are placed on the unit sphere  $x^\top P x=1$. The lines represent the intersections of quantization cells with the weighted sphere.}
    \label{fig:sphere}
\end{figure}
\begin{figure}[ht]
    \centering \includegraphics[width =0.3\textwidth]{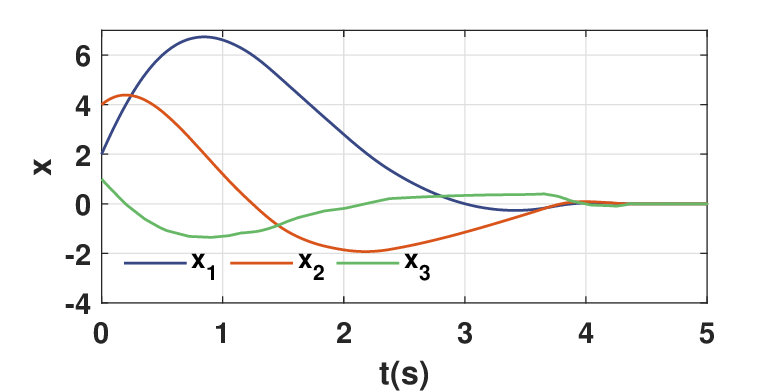}
    \caption{The system state under finite static quantization feedback.}
    \label{fig:x}
\end{figure}

\begin{figure}[ht]
    \centering \includegraphics[width =0.32\textwidth]{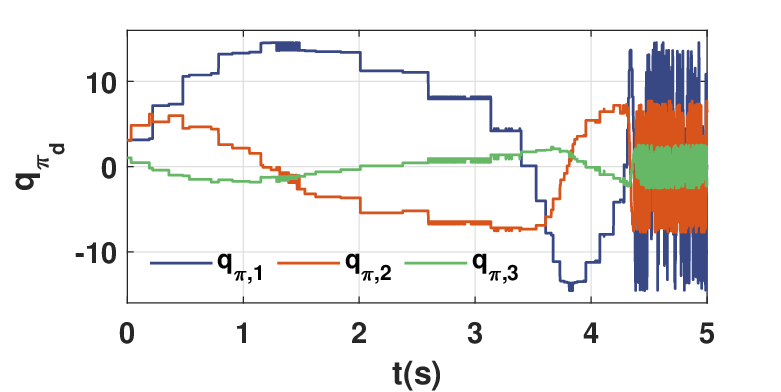}
    \caption{The quantization value $\mathfrak{q}_{\pi_\dn}(x)$.}
    \label{fig:q}
\end{figure}

\section{Conclusion}
This paper proposer a finite-time stabilizer for an LTI system with a finite static quantization of the system state. The design is based on the generalized homogeneity. The quantizer is  designed as a composition of a homogeneous projector to the unit sphere  with a spherical quantizer. In this case, the feedback law only uses discrete values on the unit sphere as quantization seeds of the state from $\R^n$. The design under limited information is addressed by solving parametric linear matrix inequalities. The chattering of the control input appears as the system state reaches the origin  due to discontinuity. Future work may focus on the optimal partition of the unit sphere and chattering attenuation.


\bibliographystyle{unsrtnat}        
\bibliography{reference.bib}           

\appendix

\end{document}